\documentclass[english,12pt]{article}

\pdfoutput=1

\usepackage[T1]{fontenc}
\usepackage[latin9]{inputenc}
\usepackage{amssymb}
\usepackage{babel}
\usepackage{amsthm}
\usepackage{verbatim}
\usepackage{amsrefs}
\usepackage{amsmath}
\usepackage{listings}
\usepackage{fullpage}
\usepackage{graphicx}
\usepackage{qtree}
\usepackage[usenames,dvipsnames]{color}
\usepackage{array}
\usepackage{enumerate}
\usepackage{polynom}
\usepackage{tikz}
\usepackage{fancyhdr}
\usepackage{etoolbox}
\usepackage{algorithm}
\usepackage{algorithmic}

\newtheorem{defin}{Definition}
\newtheorem{prop}{Proposition}

\newcommand{\p}[1]{\left(#1\right)}
\newcommand{\st}[1]{\left\{#1\right\}}

\newcommand{\bk}[1]{\left[#1\right]}

\newcommand{\ceil}[1]{\left\lceil#1\right\rceil}

\newcommand{\quot}[1]{``#1''}

\newcommand{\limf}[3]{\lim_{#1\rightarrow#2}{#3}}

\newcommand{\limi}[2]{\limf{#1}{\infty}{#2}}

\newcommand{\lb}{\hspace*{\fill}}

\DeclareMathOperator{\modd}{mod}

\newcommand{\namehead}[3]{
\lstset{breaklines=true, morecomment=[l]{//}, frame=single, showstringspaces=false, numbers=left}
\begin{flushright}
Nathan Fox\\
#2\\
#3\\
\end{flushright}
\ifstrequal{#1}{.}{}{
\begin{center}
{\Large Homework #1}
\end{center}}
}
\usepackage{thmtools}
\usepackage{thm-restate}

\renewcommand{\p}[1]{(#1)}
\newcommand{\pb}[1]{\left(#1\right)}

\newcommand{\seq}{\pb}

\newcommand{\qpr}{positive-recurrent}

\newcommand{\fora}[2]{a^{\ssc{#1}}_{#2}}
\newcommand{\forab}[2]{\fora{#1}{#2}}
\newcommand{\forb}[2]{b^{\ssc{#1}}_{#2}}
\newcommand{\forba}[2]{\forb{#1}{#2}}
\newcommand{\kn}{k}
\newcommand{\kk}{d}
\newcommand{\nii}{n}
\newcommand{\ssc}[1]{\p{#1}}
\newcommand{\fa}[1]{\seq{\fora{#1}{\kn}}}
\newcommand{\fb}[1]{\seq{\forb{#1}{\kn}}}

\begin{document}
%
%
\title{Finding Linear-Recurrent Solutions to Hofstadter-Like Recurrences Using Symbolic Computation}
\author{Nathan Fox\footnote{Department of Mathematics, Rutgers University, Piscataway, New Jersey,
\texttt{fox@math.rutgers.edu}
}}
\date{}

\maketitle

\begin{abstract}
The Hofstadter $Q$-sequence, with its simple definition, has defied all attempts at analyzing its behavior.  Defined by a simple nested recurrence and an initial condition, the sequence looks approximately linear, though with a lot of noise.  But, nobody even knows whether the sequence is infinite.  In the years since Hofstadter published his sequence, various people have found variants with predictable behavior.  Oftentimes, the resulting sequence looks complicated but provably grows linearly.  Other times, the sequences are eventually linear recurrent.  Proofs describing the behaviors of both types of sequence are inductive.  In the first case, the inductive hypotheses are fairly ad-hoc, but the proofs in the second case are highly automatable.  This suggests that a search for more sequences like these may be fruitful.  In this paper, we develop a step-by-step symbolic algorithm to search for these sequences.  Using this algorithm, we determine that such sequences come in infinite families that are themselves plentiful.  In fact,there are hundreds of easy to describe families based on the Hofstadter $Q$-recurrence alone.
\end{abstract}

\section{Introduction}
The Hofstadter $Q$-sequence, which is defined by the recurrence $Q\p{\nii}=Q\p{\nii-Q\p{\nii-1}}+Q\p{\nii-Q\p{\nii-2}}$ and the initial conditions $Q\p{1}=Q\p{2}=1$, was first introduced by Douglas Hofstadter in the 1960s~\cite{geb}.  The first fifteen terms increase monotonically, but thereafter the sequence rapidly devolves into chaos.  On the macro level, $Q\p{\nii}$ seems to oscillate around $\frac{\nii}{2}$, but nobody has been able to prove anything more than the statement that if
\[
\limi{\nii}{\frac{Q\p{\nii}}{\nii}}
\]
exists, it must equal one half~\cite{golomb}.  Furthermore, nobody has been able to prove that $Q\p{\nii}$ even \emph{exists} for all $\nii$.  If $Q\p{\nii-1}\geq \nii$ for some $\nii$, then $Q\p{\nii}$ would be defined in terms of $Q\p{k}$ for some $k\leq 0$.  But, $Q$ is not defined on nonpositive inputs, so $Q\p{\nii}$ would fail to exist.  All subsequent terms would also fail to exist, so the sequence would be finite in this scenario.  If a sequence is finite because of this sort of happenstance, we say that the sequence \emph{dies}.  We do not know whether the Hofstadter $Q$-sequence dies, but we do know that it exists for at least $10^{10}$ terms~\cite[A005185]{oeis}.

More recently, Conolly examined the recurrence $C\p{\nii}=C\p{\nii-C\p{\nii-1}}+C\p{\nii-1-C\p{\nii-2}}$ with $C\p{1}=C\p{2}=1$ as the initial condition~\cite{con}.  This setup looks quite similar to the $Q$-recurrence; the only difference is the $-1$ in the second term.  But, the resulting sequence grows monotonically and satisfies
\[
\limi{\nii}{\frac{C\p{\nii}}{\nii}}=\frac{1}{2}.
\]
In particular, $C\p{\nii}-C\p{\nii-1}$ is always either $0$ or $1$, a property commonly known as \emph{slow}.  Similar sequences have been studied by other authors~\cite{tanny}.  Most of the sequences they have analyzed use recurrences with shifts similar to the one in Conolly's recurrence, though they also analyze the Hofstadter $V$-Sequence, given by $V\p{\nii}=V\p{\nii-V\p{\nii-1}}+V\p{\nii-V\p{\nii-4}}$ with $V\p{1}=V\p{2}=V\p{3}=V\p{4}=1$, which is also slow~\cite{hofv}.  Another example of a slow sequence is the Hofstadter-Conway \$10000 Sequence, given by $A\p{\nii}=A\p{A\p{\nii-1}}+A\p{\nii-A\p{\nii-1}}$ with $A\p{1}=A\p{2}=1$.  Conway famously offered a \$10000 prize for a sufficient analysis of the behavior of this sequence, which was claimed by Colin Mallows a few years later~\cite{mallows}.

There are two methods commonly used to prove that Hofstadter-like sequences are slow.  In some cases, there are combinatorial interpretations for slow sequences involving counting leaves in nested tree structures~\cite{isgur2}.  The sequence counting the leaves is obviously slow; the main difficulty comes in showing that the nested recurrence also describes the same structure.  For some slow sequences, though, there is no known combinatorial interpretation.  The other proofs of slowness usually go by induction with complicated inductive hypotheses.  For a sequence $\seq{a_n}$, one would love to work with just the inductive hypothesis that $a_m-a_{m-1}\in\st{0,1}$ for all $m<n$, but this is never enough.  Instead, additional inductive hypotheses are required to handle certain cases.  These extra hypotheses strongly depend on the sequence in question.  While the \quot{shifted} sequences have similar proofs to each other, the proof for the Hofstadter $V$-sequence uses a different sort of inductive hypothesis.  One would like to try to automate these proofs, but this would require some method of determining the appropriate inductive hypotheses.

While the study of slow sequence has been fruitful~\cite{erickson,isgur1,isgur2,tanny}, other predictable Hofstadter-like sequences have been found.  Whereas slow sequences often result from varying the recurrence while retaining the initial condition of the $Q$-sequence, sequences from this other class often result from varying the initial condition while retaining the original $Q$-recurrence.  The first such example comes from Golomb~\cite{golomb}, who replaces the initial conditions of the $Q$-sequence with the new values $Q\p{1}=3$, $Q\p{2}=2$, $Q\p{3}=1$.  (We will use the shorthand $\bk{3,2,1}$ to describe this initial condition, and we will use similar notation going forward.)  The first few terms of the resulting sequence are $3,2,1,3,5,4,3,8,7,3,11,10,\ldots$, and the pattern continues forever, giving the linear-recurrent (in fact, quasilinear) sequence defined by
\[
\begin{cases}
Q\p{3\kn}=3\kn-2\\
Q\p{3\kn+1}=3\\
Q\p{3\kn+2}=3\kn+2.
\end{cases}
\]

In our current setup, the only linear recurrent solutions we can possibly obtain to the Hofstadter $Q$-recurrence would be quasilinear, since we require $Q\p{\nii}\leq \nii$ for all $\nii$.  For our purposes, this condition is overly restrictive.  To allow ourselves more flexibility, we will instead say that $Q\p{\nii}=0$ whenver $\nii\leq0$.  The original question of the existence of the Hofstadter $Q$-sequence can still be asked: \quot{Is $Q\p{\nii}\leq \nii$ for all $\nii$?}  Note that it is still possible for a sequence to die in this new setting.  For example, in the Hofstadter $Q$-sequence under this convention, if it happens that \emph{both} $Q\p{\nii-1}\geq \nii$ and $Q\p{\nii-2}\geq \nii$, then we obtain $Q\p{\nii}=0$.  As a result, $Q\p{\nii+1}$ would be defined in terms of itself, and hence it would be undetermined.

This more general setting allows us to find linear-recurrent solutions to the Hofstadter $Q$-recurrence that were not possible before.  One of these was found by Ruskey, who used the initial condition 
$\bk{3,6,5,3,6,8}$ 
to embed the Fibonacci numbers in the Hofstadter $Q$-recurrence~\cite{rusk}.  This sequence takes values
\[
\begin{cases}
Q\p{3\kn}=F_{\kn+4}\\
Q\p{3\kn+1}=3\\
Q\p{3\kn+2}=6.
\end{cases}
\]

The solutions of Golomb and Ruskey both satisfy linear recurrences ($a_n=2a_{n-3}-a_{n-6}$ in the case of Golomb; $a_n=2a_{n-3}-a_{n-9}$ in the case of Ruskey).  We would like to develop a general method for finding more solutions like these.  One might hope for a method that, given a linear recurrence and a Hofstadter-like recurrence, determines whether there is a sequence that eventually satisfies both of them.  Unfortunately, this is quite a lofty goal.  Rather, we exploit a deeper structure of these sequences and generalize that.  Golomb's sequence is a quasipolynomial with period $3$.  Ruskey's sequence, while not a quasipolynomial, is also structured as an interleaving of three simpler sequences.  We will describe an automatic way of searching for such interleaved solutions, where the following items are allowed to vary:
\begin{itemize}
\item The recurrence under consideration
\item The number of interleaved sequences
\item The growth rate of each of the interleaved sequences.
\end{itemize}
The methods of this paper apply to \emph{linear} nested recurrences.  These are recurrences of the form
\[
Q\p{\nii}=P\p{\nii}+\sum_{i=1}^{\kk}\alpha_i Q\p{E_i},
\]
where $P\p{\nii}$ is an explicit expression in $\nii$, $k$ is a nonnegative integer, each $\alpha_i$ is an integer, and each $E_i$ is an expression of the same form as the generic formula for $Q\p{\nii}$ (thereby allowing for arbitrarily many nesting levels).  The methods will not apply completely in all cases; certain positivity conditions will need to be satisfied.  Most commonly, we will have $P\p{\nii}=0$, each $\alpha_i$ positive, and each $E_i$ of the form $n-\beta_i-Q\p{\nii-\gamma_i}$ for some nonnegative integer $\beta_i$ and positive integer $\gamma_i$.  (This is the case for the Hofstadter $Q$-recurrence, where $P\p{\nii}=0$, $k=2$, $\alpha_1=\alpha_2=1$, $E_1=\nii-Q\p{\nii-1}$, and $E_2=\nii-Q\p{\nii-2}$.)  We will call recurrences satisfying this condition about the $E_i$'s \emph{basic}, and our methods will always apply fully to basic recurrences.

In Section~\ref{sec:prs}, we will introduce a formalism that encapsulates the notion of an interleaving of simple sequences.  Then, in Section~\ref{sec:soft}, we will describe our algorithm that finds these special solutions.  Finally, in Sections~\ref{sec:find} and~\ref{sec:fut}, we will describe some notable sequences found using the methods in this paper, and we will discuss some future extensions.

A Maple package implementing all of the algorithms in this paper, as well as some related procedures, can be found at \texttt{http://math.rutgers.edu/$\sim$nhf12/nicehof.txt}.  Generally speaking, the procedures in this package offer more general versions of the algorithms described in this paper.  For example, the code allows the user to modify the \quot{default} initial conditions (for when $\nii\leq0$).
\section{Positive Recurrence Systems}\label{sec:prs}
In the case of Ruskey's solution, each of the three interleaved sequences can be described by a homogeneous linear recurrence with nonnegative coefficients:
\[
\begin{cases}
a_{\kn}=a_{\kn-1} & a_0=3\\
b_\kn=b_{\kn-1} & b_0=6\\
c_\kn=c_{\kn-1}+c_{\kn-2} & c_1=5, c_2=8.
\end{cases}
\]
(Note that these recurrences are not unique.)
Golomb's solution cannot be expressed in this way,
but each of its interleaved sequences can be described by a nonhomogeneous linear recurrence
with nonnegative coefficients:
\[
\begin{cases}
a_\kn=3+a_{\kn-1} & a_0=3\\
b_\kn=3\\
c_\kn=3+c_{\kn-1} & c_1=1.
\end{cases}
\]
In both of these cases, we have a system of three nonhomogneous linear recurrences where all coefficients are nonnegative.  Here, none of the recurrences refer to each other in their definitions, but theoretically this should be possible.
This leads to the following generalization:
\begin{defin}\label{def:prs}
A \emph{positive recurrence system} is a system of $m$ nonhomogeneous linear recurrences of the form
\[
\begin{cases}
\forab{1}{\kn}=P_1\p{\kn}+\sum\limits_{\ell=1}^{\kk}\sum\limits_{j=1}^m\alpha_{1,\ell,j} \forab{j}{\kn-\ell}\\
\forab{2}{\kn}=P_2\p{\kn}+\sum\limits_{\ell=1}^{\kk}\sum\limits_{j=1}^m\alpha_{2,\ell,j} \forab{j}{\kn-\ell}\\
\vdots\\
\forab{m}{\kn}=P_m\p{\kn}+\sum\limits_{\ell=1}^{\kk}\sum\limits_{j=1}^m\alpha_{m,\ell,j} \forab{j}{\kn-\ell}
\end{cases}
\]
satsifying the following conditions:
\begin{itemize}
\item $\kk$ is a nonnegative integer.
\item $P_1$ through $P_m$ are eventually nonnegative integer-valued polynomials.
\item Each $\alpha_{i,\ell,j}$ is a nonnegative integer.
\end{itemize}
\end{defin}

Note that, for convenience, we may sometimes have a recurrence system where $\forab{i}{\kn}$ refers to $\forab{j}{\kn}$ for some $j<i$.  This is permissible, as we can just replace $\forab{j}{\kn}$ with its right-hand side in order to conform to Definition~\ref{def:prs}.

The solutions to Hofstadter-like recurrences that we seek will
eventually be interleavings of sequences that together satisfy a positive recurrence system.  What follows is a formalization of this notion:
\begin{defin}
Let $m$ be a positive integer.  The sequence $\seq{a_\kn}_{\kn\geq1}$ is \emph{\qpr} with period $m$ if there exist an integer $K$ such that the sequences $\st{\seq{a_{ms+r}}_{s\geq K}: 0\leq r<m}$ satisfy a positive recurrence system.
\end{defin}
Observe that eventually nonnegative polynomials are trivially {\qpr} with period $1$, as we can take $\kk=0$.  Also, any sequence satisfying a homogeneous linear recurrence with nonnegative coefficients is {\qpr} with period $1$, as we can take $P_1$ identically $0$.

Generally speaking, any positive-recurrent sequence is eventually linear recurrent, as are each of the component sequences.  This is true because a positive recurrence system can be converted into a linear system of equations for the generating functions of the component sequences.  Each generating function is therefore a rational function.

We will be concerned with determining the rate of growth of each sequence
in a solution to a
positive recurrence system.  
In order for things to be well-defined and easy to analyze, we will need the following technical definition.
\begin{defin}
An initial condition of length $N$ to a positive recurrence system is called \emph{eventually positive} if the following conditions hold:
\begin{itemize}
\item If $\kn\geq N$, then $P_r\p{\kn}\geq0$ for all $r$ and all $\kn$.
\item For all $0\leq i\leq \kk$, $\forab{r}{N-i}>0$ for all $r$.
\end{itemize}
\end{defin}

Any long enough positive initial condition is eventually positive.  But, this definition allows for some nonpositive values early in the initial condition, so long as those values are never used in calculating recursively defined terms.  Furthermore, we require all of the polynomials to be nonnegative when calculating recursive terms.  This will be useful in our analysis, though it is not strictly necessary.  (A much more complicated, weaker condition would suffice, and, in that case, we would not even need all of the polynomials to be eventually nonnegative.)

In the case where we have a solution to a positive recurrence system given by an eventually positive initial condition, the following algorithm determines the order of growth of each component sequence.
\begin{enumerate}
\item \label{it:G}Define a weighted directed graph $G$ as follows:
\begin{itemize}
\item The vertices of $G$ are the integers $\st{1,\ldots,m}$.
\item There is an arc from $i$ to $j$ if and only if, for some $\ell$, $\alpha_{i,\ell,j}>0$.
\item The weight of the arc from $i$ to $j$ is
\[
\sum_{\ell=1}^{\kk}\alpha_{i,\ell,j}.
\]
\end{itemize}
\item Initialize variables $d_1, d_2,\ldots, d_m$ so that $d_i$ equals the degree of $P_i$.
\item\label{st:inf} Let $W$ denote the set of vertices $v$ in $G$ satisfying one of the following:
\begin{itemize}
\item $v$ is in a directed circuit with at least one arc having weight greater than $1$.
\item $v$ is in more than one directed circuit that does not contain $v$ as an intermediate vertex.
\end{itemize}
For each $v\in W$, set $d_v$ to $\infty$ and delete any outgoing arc from $v$ in $G$ that is part of a cycle.  Call the resulting graph $G'$.  (We can actually delete \emph{all} outgoing arcs from $v$, but the form we have stated here will be more useful when we prove this algorithm's correctness.)
\item Define the following relation $\sim$ on $\st{1,2,\ldots,m}$:
\[
i\sim j\text{ if and only if }\p{i=j}\text{ or }\p{i\text{ and }j\text{ are in a cycle together in }G'}.
\]
It is easy to check that $\sim$ is an equivalence relation.  Each equivalence class is either a single vertex or a cycle.  (Every vertex in a cycle is in exactly one cycle, or else we would have removed all of its cycle-making outgoing arcs in step~\ref{st:inf}, and all such vertices are equivalent to each other.)
\item Define a directed graph $H$ as follows:
\begin{itemize}
\item The vertices of $H$ are the equivalence classes of $\st{1,2,\ldots,m}$ under $\sim$.
\item There is an arc from class $I$ to class $J$ if and only if there is an arc in $G'$ from some $i\in I$ to some $j\in J$.
\end{itemize}
If $H$ contains a directed cycle $I_1,I_2,\ldots,I_q$, then for each $1\leq h<q$, there is an arc in $G'$ from some $i^{out}_h\in I_h$ to some $i^{in}_{h+1}\in I_{h+1}$.  Also, there is an arc in $G'$ from some $i^{out}_q\in I_q$ to some $i^{in}_1\in I_1$.  Furthermore, by the definition of $\sim$, for each $h$, there is a (possibly trivial) directed path from $i^{in}_h$ to $i^{out}_h$ within $I_h$.  Concatenating all of these arcs together gives a cycle in $G'$ that includes elements of multiple equivalence classes, which contradicts the definition of $G'$.

So, we can conclude that $H$ contains no directed cycles.
\item For each vertex $I$ of $H$, initialize a variable $d_I=\max\limits_{i\in I}d_i$.
\item\label{st:ts} Topologically sort the vertices of $H$.  Consider the vertices $I$ from last to first:
\begin{itemize}
\item If $I$ is a cycle in $G'$ (including a single vertex with a loop), set $d_I$ to $d_I+1$.
\item For all $J$ with an arc from $J$ to $I$, set $d_J=\max\pb{d_J, d_I}$.
\end{itemize}
\end{enumerate}
At the end of this process, we have values $d_I$ for each equivalence class $I$.  In general, for an integer $r$, let $\bar{r}$ denote its equivalence class under $\sim$.  We now make the following claim:
\begin{restatable}{claim}{clmain}
\label{cl:main}
Suppose we have a positive recurrence system with $m$ component sequences, along with an eventually positive initial condition.  Let $1\leq r\leq m$ be an integer. 
If $d_{\bar{r}}<\infty$, then $\forab{r}{\kn}=\Theta\pb{\kn^{d_{\bar{r}}}}$.  If $d_{\bar{r}}=\infty$, then $\forab{r}{\kn}$ grows exponentially.
\end{restatable}
The proof of this claim is uninteresting, so it has been relegated to Appendix~\ref{sec:pfclmain}.  The general strategy is to inductively substitute expressions around cycles in order to obtain recurrences for sequences $\forab{r}{\kn}$ that refer only to earlier versions of themselves.  Then, we use the general theory of linear recurrences to make claims about their asymptotics.
\section{The Algorithm}\label{sec:soft}
What follows is a description of a generic run of our algorithm for finding positive-recurrent solutions to nested recurrences.  We will walk through the steps of searching for period-$m$ solutions to a linear nested recurrence $Q\p{\nii}$.
A dominant term in the complexity of our algorithm is that it is exponential in $m$, so the first thing we do is fix $m$.
We will follow the steps of the algorithm by applying it to a running example: searching for period-$4$ solutions to the recurrence
\[
R\p{\nii}=R\p{\nii-R\p{\nii-1}}+R\p{\nii-R\p{\nii-2}}+R\p{\nii-R\p{\nii-3}}.
\]
We choose this particular example because it will illustrate many facets of the algorithm without becoming too unwieldy.
\subsection{Fixing the Behavior of the Subsequences}
A positive-recurrent solution to $Q\p{\nii}$ with period $m$ has the form
\[
\begin{cases}
Q\p{m\kn}=\fora{0}{\kn}\\
Q\p{m\kn+1}=\fora{1}{\kn}\\
Q\p{m\kn+2}=\fora{2}{\kn}\\
\hspace{0.5in}\vdots\\
Q\p{mk+\p{m-1}}=\fora{m-1}{\kn}
\end{cases}
\]
for some sequences $\fa{0}$ through $\fa{m-1}$.  
(For convenience, we index the interleaved sequences from zero in this context.)  
For convenience, we define the following growth properties that these component sequences may have:
\begin{defin}\lb
\begin{itemize}
\item We say $\fa{r}$ is \emph{constant} if, for sufficiently large $\kn$, $\fora{r}{\kn}=A$ for some constant $A$.
\item We say $\fa{r}$ is \emph{linear} if, for sufficiently large $\kn$, $\fora{r}{\kn}=A\kn+B$ for some constants $A$ and $B$.
\item We say $\fa{r}$ is \emph{superlinear} if $\fora{r}{\kn}=\omega\p{\kn}$.
\item We say that $\fa{r}$ is \emph{standard linear} if $\fora{r}{\kn}=m\kn+B$ for some constant $B$.
\item We say that $\fa{r}$ is \emph{steep linear} if $\fora{r}{\kn}=A\kn+B$ for some constants $A$ and $B$ with $A>m$.
\item We say that $\fa{r}$ is \emph{steep} if $\fa{r}$ is either steep linear or superlinear.
\end{itemize}
\end{defin}
To start the algorithm, we need to decide, for each of the $m$ component sequences, are we looking for a solution where that subsequence is  \emph{constant}, \emph{standard linear}, or \emph{steep}?
(We do not concern ourselves with component sequences intermediate in growth between constant and standard linear, as they seem to be uncommon in this context and can be harder to analyze.  But, the Maple package can sometimes handle these, as long as the user explicitly asks it to.)
To keep track of our choices, the algorithm stores variables $A_0,A_1,\ldots,A_{m-1}$.
We set $A_r=0$ if we decide that $\fa{r}$ is to be constant, $A_r=m$ if $\fa{r}$ is standard linear, and $A_r=\infty$ if $\fa{r}$ is steep.  Observe that if $\fa{r}$ is not steep, then $\fa{r}=A_r\kn+B_r$ for some constant $B_r$.
In general, to perform an exhaustive search for positive-recurrent solutions, we iterate through the $3^m$ possible overall behaviors.

In our running example, we seek a solution of the form
\[
\begin{cases}
R\p{4\kn}=\fora{0}{\kn}\\
R\p{4\kn+1}=\fora{1}{\kn}\\
R\p{4\kn+2}=\fora{2}{\kn}\\
R\p{4\kn+3}=\fora{3}{\kn}.
\end{cases}
\]
Going forward, we will assume we are treating the case $A_0=\infty$, $A_1=4$, $A_2=0$, and $A_3=0$.  That is, we are seeking a solution with $\fa{0}$ steep, $\fa{1}$ standard linear, and the other two sequences constant.
\subsection{Unpacking the Recurrence}
Now that we have fixed the forms of the $\fa{r}$, we can use these forms inductively to \quot{unpack} the recurrent form for each $Q\p{m\kn+r}$.  (The base case of the induction is covered by the fact that we allow an arbitrarily long initial condition.)  We start from the innermost calls to $Q$ (i.e., calls without another $Q$ inside them) and work our way outward, eliminating (nearly) all of the the $Q$'s with $\kn$'s inside them.  For an expression of the form $Q\p{m\kn+r}$, we replace it by the expression $A_r\kn+B_r$, where $A_r$ is the variable tracking the growth of $\fa{r}$ and $B_r$ is a symbol.  
Each step after rewriting the inner calls involves looking at an expression of the form $Q\p{c}$, $Q\p{-\infty\kn-c}$, or $Q\p{m\kn-c}$ for some, possibly symbolic, constant $c$.  Expressions of the first type are constant; expressions of the second type are zero (by our convention that evaluation at nonpositive indices gives zero).  In both of these cases, we can immediately move outward.  Expressions of the third type equal 
$\forab{r}{k-t}$ for $r=\pb{-c}\modd m$ and $t=\ceil{\frac{c}{m}}$.  
This all depends on the congruence class of $c$ mod $m$, since
\[
\ceil{\frac{c}{m}}=\begin{cases}
\frac{c}{m} & c\equiv0\pb{\modd m}\\
\frac{c+m-\pb{c\modd m}}{m} & \text{otherwise}.
\end{cases}
\]
So, at this step, we iterate through all possible congruence classes of these constants.

Once we do this, we replace $Q\p{m\kn-c}$ by $A_r\p{\kn-t}+B_r$ ( $r=\pb{-c}\modd m$ and $t=\ceil{\frac{c}{m}}$) for symbolic constant $B_r$.  In the case the $\fa{r}$ is steep, this replacement will ensure that the coefficient on $\kn$ in $m\kn-Q\p{m\kn-c}$ will be $-\infty$.  If $\fa{r}$ is not steep, this simply substitutes the formula for $\fa{r}$ into the recurrence.
In the case that we are considering an outermost $Q$ and $\fa{r}$ is steep, we do not replace the call to $\fa{r}$; we will leave it unevaluated.

If $Q\p{\nii}$ is basic, then the congruence choices we will have to make in this step are easily predetermined.
\begin{prop}\label{prop:stdb}
If $Q\p{\nii}$ is basic, then the $c$ values appearing in expressions $Q\p{m\kn-c}$ that appear at the outer level when unpacking the recurrence are precisely determined by fixing the values of the $B_r$'s for which $\fa{r}$ is constant.
\end{prop}
\begin{proof}
If the recurrence for $Q\p{\nii}$ is basic, it is of the form
\[
Q\p{\nii}=P\p{\nii}+\sum_{i=1}^d\alpha_iQ\p{\nii-\beta_i-Q\p{\nii-\gamma_i}}.
\]
Substituting $\nii=m\kn+r$, we obtain
\[
Q\p{m\kn+r}=P\p{m\kn+r}+\sum_{i=1}^d\alpha_iQ\p{m\kn+r-\beta_i-Q\p{m\kn+r-\gamma_i}}.
\]
The innermost expressions are the expressions $Q\p{m\kn+r-\gamma_i}$.  For notational ease, let $r_i=\p{r-\gamma_i}\modd m$.  There are three cases to consider:
\begin{description}
\item[$Q\p{m\kn+r-\gamma_i}$ is constant:] In this case, $Q\p{m\kn+r-\gamma_i}=B_{r_i}$.  We then see the expression $Q\p{m\kn-\beta_i-B_{r_i}}$ at the outer level.  Since $\beta_i$ is a predetermined constant, the constant $\beta_i+B_{r_i}$ that appears here is completely determined by fixing $B_{r_i}$.  Furthermore, $\fa{r_i}$ is constant; it eventually equals $B_{r_i}$.
\item[$Q\p{m\kn+r-\gamma_i}$ is standard linear:]  In this case, $Q\p{m\kn+r-\gamma_i}=m\kn+B_{r_i}$.  This leads to the expression $Q\p{-\beta_i-B_{r_i}}$ at the outer level, which is a constant.
\item[$Q\p{m\kn+r-\gamma_i}$ is steep:] In this case, the expression that appears at the outer level in this term is automatically zero.
\end{description}
In summary, only one case results in an outer expression of the form $Q\p{m\kn-c}$, and the value of $c$ in that case is completely determined by the value of $B_r$ for an $r$ with $\fa{r}$ constant, as required.
\end{proof}

The net result of this step is that we can now read off a recurrence system for the component sequences.  If the recurrence was basic, we claim that this will be a positive recurrence system. First, all coefficients in the homogeneous part of the system will be nonnegative, because the coefficients on the recursive calls in a basic recurrence are all positive.  Also, all constant component sequences will be positive, since we cannot have a solution with infinitely many nonpositive entries (or else we would not be able to explicitly calculate terms).  Beyond this, nonhomogeneous parts are built out of recursive calls, which will inductively result in eventually nonnegative polynomials.

In general, in order to guarantee the correctness of our algorithm, we need the nested recurrence to be such that we obtain a positive recurrence system in this step.  We stated earlier that the eventual nonnegativity condition on the nohnomogeneous parts is overly restrictive for what we want to do with positive recurrence systems, so it makes sense to continue to run the rest of this algorithm even without a positive recurrence system.  The algorithm may still succeed, but we can no longer guarantee that it will succeed.


In our running example, we are concerned with the congruence classes of $B_2$ and $B_3$ mod $4$.  There are eight cases to check.  Going forward, we will assume we are looking at the case where $B_2\equiv0\pb{\modd 4}$ and $B_3\equiv3\pb{\modd 4}$.  Under these assumptions, here is how the recurrence unpacks:
{\allowdisplaybreaks
\begin{align*}
R\p{4\kn}&=R\p{4\kn-R\p{4\kn-1}}+R\p{4\kn-R\p{4\kn-2}}+R\p{4\kn-R\p{4\kn-3}}\\
&=R\!\pb{4\kn-\fora{3}{\kn-1}}+R\!\pb{4\kn-\fora{2}{\kn-1}}+R\!\pb{4\kn-\fora{1}{\kn-1}}\\
&=R\p{4\kn-0\pb{\kn-1}-B_3}+R\p{4\kn-0\pb{\kn-1}-B_2}+R\p{4\kn-4\pb{\kn-1}-B_1}\\
&=R\p{4\kn-B_3}+R\p{4\kn-B_2}+R\p{4-B_1}\\
&=\fora{1}{\kn-\frac{B_3+1}{4}}+\fora{0}{\kn-\frac{B_2}{4}}+R\p{4-B_1}\\
&=4\pb{\kn-\frac{B_3+1}{4}}+B_1+\fora{0}{\kn-\frac{B_2}{4}}+R\p{4-B_1}\\
&=4\kn-B_3-1+B_1+\fora{0}{\kn-\frac{B_2}{4}}+R\p{4-B_1}.
\end{align*}
\begin{align*}
R\p{4\kn+1}&=R\p{4\kn+1-R\p{4\kn}}+R\p{4\kn+1-R\p{4\kn-1}}+R\p{4\kn+1-R\p{4\kn-2}}\\
&=R\!\pb{4\kn+1-\fora{0}{\kn}}+R\!\pb{4\kn+1-\fora{3}{\kn-1}}+R\!\pb{4\kn+1-\fora{2}{\kn-1}}\\
&=R\p{4\kn+1-\infty \kn-B_0}+R\p{4\kn+1-0\pb{\kn-1}-B_3}+R\p{4\kn+1-0\pb{\kn-1}-B_2}\\
&=R\p{-\infty \kn+1-B_0}+R\p{4\kn+1-B_3}+R\p{4\kn+1-B_2}\\
&=0+\fora{2}{\kn-\frac{B_3+1}{4}}+\fora{1}{\kn-\frac{B_2}{4}}\\
&=0\pb{\kn-\frac{B_3+1}{4}}+B_2+4\pb{\kn-\frac{B_2}{4}}+B_1\\
&=B_2+4\kn-B_2+B_1\\
&=4\kn+B_1.
\end{align*}
\begin{align*}
R\p{4\kn+2}&=R\p{4\kn+2-R\p{4\kn+1}}+R\p{4\kn+2-R\p{4\kn}}+R\p{4\kn+2-R\p{4\kn-1}}\\
&=R\!\pb{4\kn+2-\fora{1}{\kn}}+R\!\pb{4\kn+2-\fora{0}{\kn}}+R\!\pb{4\kn+2-\fora{3}{\kn-1}}\\
&=R\p{4\kn+2-4\kn-B_1}+R\p{4\kn+2-\infty \kn-B_0}+R\p{4\kn+2-0\pb{\kn-1}-B_3}\\
&=R\p{2-B_1}+R\p{-\infty \kn+2-B_0}+R\p{4\kn+2-B_3}\\
&=R\p{2-B_1}+0+\fora{3}{\kn-\frac{B_3+1}{4}}\\
&=R\p{2-B_1}+0\pb{\kn-\frac{B_3+1}{4}}+B_3\\
&=R\p{2-B_1}+B_3.
\end{align*}
\begin{align*}
R\p{4\kn+3}&=R\p{4\kn+3-R\p{4\kn+2}}+R\p{4\kn+3-R\p{4\kn+1}}+R\p{4\kn+3-R\p{4\kn}}\\
&=R\!\pb{4\kn+3-\fora{2}{\kn}}+R\!\pb{4\kn+3-\fora{1}{\kn}}+R\!\pb{4\kn+3-\fora{0}{\kn}}\\
&=R\p{4\kn+3-0\kn-B_2}+R\p{4\kn+3-4\kn-B_1}+R\p{4\kn+3-\infty \kn-B_0}\\
&=R\p{4\kn+3-B_2}+R\p{3-B_1}+R\p{-\infty \kn+3-B_0}\\
&=\fora{3}{\kn-\frac{B_2}{4}}+R\p{3-B_1}+0\\
&=0\pb{\kn-\frac{B_2}{4}}+B_3+R\p{3-B_1}\\
&=B_3+R\p{3-B_1}.
\end{align*}}
\subsection{Checking for Structural Consistency}
The previous step has given us a positive recurrence system that is eventually satisfied by the $\fa{r}$'s.  But, this system may not be consistent with the choices we have made.
First, we check the following things:
\begin{itemize}
\item If $\fa{r}$ is constant, our expression for $Q\p{m\kn+r}$ should consist only of constants.
\item If $\fa{r}$ is standard linear, our expression for $Q\p{m\kn+r}$ should be of the form $m\kn+c$ for some (possibly complicated) constant $c$.
\item If $\fa{r}$ is steep, our expression for $Q\p{m\kn+r}$ should have at least one of the following:
\begin{itemize}
\item A term $d\kn$ with $d>m$.
\item A reference to some steep $\fa{r'}$.
\end{itemize}
\end{itemize}
If any of these is violated for any $r$, there is no solution with the given $A$ values and congruence conditions.

If all of the above conditions are satisfied, we must determine the nature of each steep $\fa{r}$.
In particular, we need to determine if each one is steep linear or superlinear.  As a bonus, we will be able to determine the degree of $\fa{r}$ if it is a polynomial.
Since the expressions we have for the $\fa{r}$'s form a positive recurrence system, we can use the algorithm from Section~\ref{sec:prs} to accomplish precisely this task, provided we will start with an eventually positive initial condition.  (We will construct this initial condition later.)  In running this algorithm, we may find that we actually do not have a solution, as the third case above includes expressions like $\forab{r}{k}=\forab{r}{k-1}$, which do not result in steep sequences.

In our example, we verify successfully that $\fa{1}$ is standard linear (the expression we obtained for $R\p{4\kn+1}$ is $4\kn+B_1$) and that $\fa{2}$ and $\fa{3}$ are constant (expressions $R\p{2-B_1}+B_3$ and $B_3+R\p{3-B_1}$ respectively).
We now run our algorithm on the positive recurrence system we obtained:
\[
\begin{cases}
\fora{0}{\kn}=4\kn-B_3-1+B_1+\fora{0}{\kn-\frac{B_2}{4}}+R\p{4-B_1}\\
\fora{1}{\kn}=4\kn+B_1\\
\fora{2}{\kn}=R\p{2-B_1}+B_3\\
\fora{3}{\kn}=B_3+R\p{3-B_1}.
\end{cases}
\]
The graph $G$ consists of four vertices.  Vertex $0$ has a loop with weight $1$; the other three vertices are isolated.  We initialize $d_0=1$, $d_1=1$, $d_2=0$, and $d_3=0$.  Step~\ref{st:inf} doesn't affect any of the vertices, so $G'=G$.  Similarly, $\sim$ has no nontrivial relations, so $H\cong G$ via the isomorphism $i\leftrightarrow\st{i}$.  When we process vertex $\st{0}$ in $H$, we set $d_{\st{0}}=2$, and this is the only change made in Step~\ref{st:ts}.  So, we obtain that $\fa{0}$ is quadratic.
\subsection{Building a Constraint Satisfaction Problem}\label{subsec:csp}
If our parameters produce a solution, we now know precisely what the structure of that solution must be.  At this point, we need to see if a solution can actually be realized.
%

In 
order to have a solution, we must check the following:
\begin{itemize}
\item If $\fa{r}$ is constant, we must have $B_r>0$.  Otherwise, our solution would have infinitely many nonpositive values.  This is not allowed, as we would then not be able to explicitly calculate terms of the sequence.
\item If $\fa{r}$ is constant, $B_r$ must equal our expression for $Q\p{m\kn+r}$.
\item If $\fa{r}$ is 
standard 
linear, $B_r$ must equal the constant term in our expression for $Q\p{m\kn+r}$.
\item If $\fa{r}$ is steep linear, we may need a steepness constraint.  This constraint is somewhat more complicated; we describe it below.
\item Any constant that we have forced to have a certain congruence mod $m$ must actually have that congruence.
\item For any two constants of the form $Q\p{c}$ and $Q\p{d}$ that appear, if $c=d$ then $Q\p{c}$ must equal $Q\p{d}$.
\item If constant $Q\p{c}$ appears, then if $c\leq0$ we must have $Q\p{c}=0$.
\end{itemize}
The last two of these restrictions gives a set of conditional constraints to check; the rest of the constraints are unconditional constraints.

As mentioned above, constraining steep linear $\fa{r}$'s to actually be steep requires a more complicated constraint.  This stems from the fact that steep linear $\fa{r}$'s can arise in three different ways.
The steep linear $\fa{r}$'s are a subset of the linear $\fa{r}$'s.  In terms of the positive recurrence system algorithm, the linear $\fa{r}$'s are the ones whose vertices are labeled $1$ when the algorithm terminates.  The following are all the ways this could happen, in terms of the graphs $G'$ and $H$ in that algorithm.
\begin{enumerate}
\item\label{it:l1} The vertex $r$ could be labeled $1$ because the expression for $Q\p{m\kn+r}$ is a degree-$1$ polynomial.
\item\label{it:l2} The vertex $r$ could be labeled $1$ because it is not in a cycle in $G'$ and, when it came time to assign $\bar{r}$ its final label, the largest label in $H$ it pointed to was a $1$.
\item\label{it:l3} The vertex $r$ could be labeled $1$ because it is in a cycle in $G'$ and, when it came time to assign $\bar{r}$ its final label, the largest label in $H$ it pointed to was a $0$ (or it pointed to no other vertices in $H$).
\end{enumerate}
In Case~\ref{it:l1}, $\fa{r}$ is steep linear if and only if the leading coefficient of that polynomial is greater than $m$.  We already checked this in our structural consistency check, so if $\fa{r}$ is linear because of Case~\ref{it:l1}, we need no steepness constraint.  In Case~\ref{it:l2}, we have that $r$ is pointing to something else linear.  But, our unpacking step would have removed all references to standard linear $\fa{r'}$'s.  So, in Case~\ref{it:l2}, every $\fa{r'}$ that $\fa{r}$ still refers to must be steep linear.  This immediately forces $\fa{r}$ itself to be steep linear without imposing any extra constraints.

This leaves only Case~\ref{it:l3}.  In this case, $r$ is in a directed cycle in $G'$, say 
\\$r=r_0,r_1,r_2,\ldots,r_{t-1},r_t=r$.
Each of the corresponding sequences has an expression of the form
\[
\fora{r_i}{\kn}=c_i+\fora{r_{i+1}}{\kn-e_i}.
\]
Repeated substitution yields the formula
\[
\fora{r}{\kn}=\sum_{i=0}^{t-1}c_i+\fora{r}{\kn-\sum_{i=0}^{t-1}e_i}
\]
for some constants $e_i$.  We require that $\fa{r}$ be steep.  This will be accomplished if we have that
\[
\sum_{i=0}^{t-1}c_i>m\sum_{i=0}^{t-1}e_i.
\]
So, this is the steepness constraint we add in Case~\ref{it:l3}.  In particular, we arrive at the same constraint for all functions in a given equivalence class.

In our example, we obtain the following constraints:
\begin{itemize}
\item $\fa{0}$ is superlinear, so there are no constraints associated to it.
\item $\fa{1}$ is standard linear, and we have $R\p{4\kn+1}=4\kn+B_1$.
This gives us the constraint $B_1=B_1$.  (This constraint is tautological, but this is okay.)
\item $\fa{2}$ is constant, and we have $R\p{4\kn+2}=R\p{2-B_1}+B_3$.  This gives us the following constraints:
\begin{itemize}
\item $B_2>0$
\item $B_2=R\p{2-B_1}+B_3$.
\end{itemize}
\item $\fa{3}$ is constant, and we have $R\p{4\kn+3}=B_3+R\p{3-B_1}$.  This gives us the following constraints:
\begin{itemize}
\item $B_3>0$
\item $B_3=B_3+R\p{3-B_1}$.
\end{itemize}
\item Our congruence constraints are
\begin{itemize}
\item $B_2\equiv0\pb{\modd 4}$
\item $B_3\equiv3\pb{\modd 4}$.
\end{itemize}
\item Our conditional constraints are
\begin{itemize}
\item If $2-B_1=3-B_1$, then $R\p{2-B_1}$ must equal $R\p{3-B_1}$.
\item If $2-B_1\leq0$, then $R\p{2-B_1}=0$.
\item If $3-B_1\leq0$, then $R\p{3-B_1}=0$.
\end{itemize}
\end{itemize}
\subsection{Solving the Constraint Satisfaction Problem}
The recurrence solution we are seeking should exist if and only if this constraint system is satisfiable.  The system of unconditional constraints is almost an integer program.  The following modifications can turn it into an integer program:
\begin{itemize}
\item Since all variables are integers, strict inequalities of the form $x>y$ can be made loose by replacing them by the equivalent inequality $x\geq y+1$.
\item Congruence constraints can be converted to equality constraints via the introduction of auxiliary variables.  Namely, $x\equiv y\pb{\modd m}$ is the same constraint as $x=Km+y$, where $K$ is a new auxiliary variable.
\end{itemize}
Furthermore, the conditional constraints can be incorporated into the integer program.  For each constraint of the form $\pb{c=d}\Rightarrow \pb{e=f}$, we consider three cases.  (If one fails, we try the next one.)
\begin{itemize}
\item Add the constraints $c=d$ and $e=f$.
\item Add the constraint $c<=d-1$.
\item Add the constraint $c>=d+1$.
\end{itemize}
And, for each constraint of the form $\pb{c\leq d}\Rightarrow \pb{e=f}$, we consider two cases.
\begin{itemize}
\item Add the constraints $c\leq d$ and $e=f$.
\item Add the constraint $c>=d+1$.
\end{itemize}

Maple has a built-in procedure for satisfying \emph{linear} integer programs.  Since we are only considering linear nested recurrences,
the program we obtain is, in fact, linear.
Integer linear programming is an \textsc{np}-hard problem~\cite{karp}.  But, experimentally, the instances that arise in this context seem not to be very hard.  Heuristically guessing $B_r$ values that are far apart seems to make satisfying the constraints be a quick process.  In particular, the constraints we obtain are typically satisfiable unless there is some obvious reason why they should not be.

The following assignments satisfy our example constraint system:
\begin{itemize}
\item $B_1=0$
\item $B_2=4$
\item $B_3=3$
\item $R\p{2}=1\pb{=R\p{2-B_1}}$
\item $R\p{3}=0\pb{=R\p{3-B_1}}$.
\end{itemize}
This means we have the following eventual solution:
\[
\begin{cases}
R\p{4\kn}=4\kn-3-1+0+\fora{0}{\kn-1}+R\p{4}=R\p{4\kn-4}+4\kn+R\p{4}-4\\
R\p{4\kn+1}=4\kn\\
R\p{4\kn+2}=4\\
R\p{4\kn+3}=3.
\end{cases}
\]
Note that these constraints have other satisfying values, and each other satisfaction leads to another eventual solution to the recurrence.
\subsection{Constructing an Initial Condition}
We now have an eventual solution to our recurrence.  But, it is quite likely that there will need to be some initial values to make the solution exist.  These initial value requirements can arise from a number of sources:
\begin{itemize}
\item An expression of the form $Q\p{\nii-c}$ in the recurrence probably requires that the initial condition have length at least $c$.  (The only case where it would not require this is if the solution we found has $Q\p{\nii-c}=0$ for $\nii\leq c$.)
\item There are sometimes equality constraints involving $Q\p{c}$ for some constant $c$.  Unless the value required of $Q\p{c}$ is what it must be anyway in the eventual solution, the initial condition must have length at least $c$.
\item In constructing our solution, we used the fact that steep sequences $\fa{r}$ are eventually greater than $\kn+c$ for any constant $c$.  But, for fixed $c$ they can be smaller for small $\kn$.  The initial condition must be long enough to allow all the steep sequences to become sufficiently large.
\end{itemize}

The question now is, when does the initial condition end?  We try to find as generic an initial condition as possible.  This means that we try to minimize the length of the initial condition while including as many free parameters as we can.  As a result, the initial condition we find may include symbols, and some of these symbols may have (obviously satisfiable) constraints placed on them.  (To find an explicit initial condition, replace each symbol by a value satisfying any constraints on it.)  Here is an outline of a procedure for finding an initial condition.  This is a more streamlined version of the procedure in the Maple package.  That version is more complicated, as it tries to use various tricks to shorten the initial condition.
\begin{enumerate}
\item Look for the largest $c_0$ such that $Q\p{c_0}$ appears in a constraint from the constraint satisfaction problem.  Initialize the list $L$ to the symbolic list $\bk{Q\p{1},Q\p{2},\ldots,Q\p{c_0}}$.
\item The constraints involving expressions of the form $Q\p{c}$ constitute a linear system.  Solve this system for as many constants $Q\p{c}$ as possible.  (If the system is underdetermined, these solutions may be in terms of others of these constants.)  Substitute these solution values into $L$ for their respective constants.
\item\label{st:nm} Denote by $\gamma$ the maximum number such that $Q\p{n-\gamma}$ appears in the recurrence.  Generate the next $\max\pb{m,\gamma}$ terms of the recurrence $Q$ with the initial condition $L$.  If $Q$ is ever evaluated at a symbolic index, assume that this index is nonpositive
and record the resulting constraint on the symbols.
(If the sequence dies during this computation, go to step~\ref{st:bad}.)
\item\label{st:4} If $\max\pb{m,\gamma}$ terms were generated, look at the terms in the steep positions and determine which future terms directly depend on them.  If the computations involving the steep terms used to compute the later terms all result in $0$, then the terms in the steep positions are sufficiently large.  (If these terms are symbolic, this may result in inequality constraints on some constants.)  Otherwise, the steep terms are not sufficiently large, so go to Step~\ref{st:bad}.
\item If the non-steep terms agree with their eventual values, then the initial condition is complete.  In this case, return $L$ along with the constraints that were imposed on the symbols when generating the $m$ terms.
\item\label{st:bad} If not enough terms were generated, terms were not large enough, or some term did not agree with its eventual value, then extend $L$ by one term.  In that position, put the value of the eventual solution.  (If that term is a member of a steep sequence, keep that term symbolic.)  Then, forget any constraints imposed on symbols in step~\ref{st:4} and return to Step~\ref{st:nm}.
\end{enumerate}
This algorithm generates an initial condition for the eventual solution found in the previous steps.

For our example, the largest $c$ such that $R\p{c}$ appears in a constraint is $c=3$.  So, we initialize $L=\bk{R\p{1},R\p{2},R\p{3}}$.  We see that $R\p{2}=1$ and $R\p{3}=0$, so this changes $L$ to $\bk{R\p{1},1,0}$.  In this case $m=4$, and $\gamma=3$ since $R\p{n-3}$ appears in the recurrence.  We now start extending $L$:
\begin{itemize}
\item We try to compute $R\p{4}$:
\begin{align*}
R\p{4}&=R\p{4-R\p{3}}+R\p{4-R\p{2}}+R\p{4-R\p{1}}\\
&=R\p{4-0}+R\p{4-1}+R\p{4-R\p{1}}\\
&=R\p{4}+R\p{3}+R\p{4-R\p{1}}.
\end{align*}
We see that $R\p{4}$ is defined in terms of $R\p{4}$, which is not allowed.  So, the sequence dies immediately.  We failed to generate $4$ terms, so we extend $L$.  The fourth term would be part of the quadratic (steep) sequence, so we keep $R\p{4}$ symbolic.   We now have $L=\bk{R\p{1},1,0, R\p{4}}$
\item The next two times we try to compute a term of $R$, the sequence will die immediately.  This is the case because the $0$ that killed our sequence in the $R\p{4}$ case is referenced again when computing $R\p{5}$ and when computing $R\p{6}$.  The predicted values for $R\p{5}$ and $R\p{6}$ are both $4$.  So, $L$ will be extended to $\bk{R\p{1},1,0, R\p{4}, 4}$ and then to $\bk{R\p{1},1,0, R\p{4}, 4, 4}$.
\item We try to compute $R\p{7}$:
\begin{align*}
R\p{7}&=R\p{7-R\p{6}}+R\p{7-R\p{5}}+R\p{7-R\p{4}}\\
&=R\p{7-4}+R\p{7-4}+R\p{7-R\p{4}}\\
&=R\p{3}+R\p{3}+R\p{7-R\p{4}}\\
&=0+0+R\p{7-R\p{4}}\\
&=R\p{7-R\p{4}}.
\end{align*}
In order to evaluate $R\p{7-R\p{4}}$, we assume that $R\p{4}\geq7$, so this term is zero.  This gives $R\p{7}=0$.  Then, when we try to compute $R\p{8}$, we will fail.  The predicted value for $R\p{7}$ is $3$, so we extend $L$ to $\bk{R\p{1},1,0, R\p{4}, 4, 4, 3}$.
\item We try to compute $R\p{8}$:
\begin{align*}
R\p{8}&=R\p{8-R\p{7}}+R\p{8-R\p{6}}+R\p{8-R\p{5}}\\
&=R\p{8-3}+R\p{8-4}+R\p{8-4}\\
&=R\p{5}+R\p{4}+R\p{4}\\
&=4+2R\p{4}.
\end{align*}
This is a symbolic answer, but that is permitted.  The term $R\p{8}$ would be in the quadratic sequence, so this agrees with the eventual solution provided $R\p{4}$ is sufficiently large.  We now try to compute $R\p{9}$:
\begin{align*}
R\p{9}&=R\p{9-R\p{8}}+R\p{9-R\p{7}}+R\p{9-R\p{6}}\\
&=R\p{9-4-2R\p{4}}+R\p{9-3}+R\p{9-4}\\
&=R\p{5-2R\p{4}}+R\p{6}+R\p{5}\\
&=R\p{5-2R\p{4}}+4+4\\
&=8+R\p{5-2R\p{4}}.
\end{align*}
If $2R\p{4}\geq5$, we obtain $R\p{9}=8$, which is its value in the eventual solution.  If we compute $R\p{10}$, we find that it equals its predicted value of $4$ provided that $2R\p{4}\geq6$.  Then, if we compute $R\p{11}$, we find that it equals its predicted value of $3$ provided that $2R\p{4}\geq7$.  Combining all of this, we find that, for all $R\p{1}$ and all $R\p{4}\geq4$, $\bk{R\p{1},1,0, R\p{4}, 4, 4, 3}$ is an initial condition that yields the eventual solution we found in the previous step.
\end{itemize}
\section{Some Findings}\label{sec:find}
As our running example illustrates, it is possible to obtain quasi-quadratic solutions to Hofstadter-like recurrences.  This begs the question as to whether higher degree quasipolynomials are possible.  If they are, the algorithm will find them.  It turns out that there are quasipolynomial solutions to the Hofstadter $Q$-recurrence of every positive degree~\cite{gengol}.  In addition, there are solutions that include both quadratic and exponential subsequences, such as the sequence obtained from the Hofstadter $Q$-recurrence with $\bk{9,0,0,0,7,9,9,10,4,9,9,3}$ as the initial condition~\cite[A275153]{oeis}.  
It is likely that a construction similar to the one for arbitrary degree quasipolynomials~\cite{gengol} will also lead to examples including higher degree polynomials along with exponentials.  There are also solutions to the Hofstadter $Q$-recurrence with linear subsequences with slopes greater than $1$, and such subsequences can be obtained by any of the three ways mentioned in Subsection~\ref{subsec:csp}:
\begin{itemize}
\item The length-$45$ initial condition
\begin{align*}
[&0, 4, -40, -9, 8, -8, 7, 1, 5, 13, -24, -1, 8, 8, 8, 1, 5, 13, -8, 7, 8, 8, 
23, 1, 5, 13, 8, 15, \\
&8, 16, 31, 1, 5, 13, 24, 23, 8, 24, 39, 1, 5, 13, 40, 31,
8]
\end{align*}
leads to a period-$8$ solution with $Q\p{8n+3}=16n-40$.  This is the case because unpacking $Q\p{8n+3}$ involves adding two standard linear terms together~\cite[A275361]{oeis}.
\item The length-$16$ initial condition
\[
[-9, 2, 9, 2, 0, 7, 9, 10, 3, 0, 2, 9, 2, 9, 9, 9]
\]
leads to a period-$9$ solution where $Q\p{9n+2}$ and $Q\p{9n+8}$ both have slope $10$.  But, $Q\p{9n+2}$ has slope $10$ because unpacking it yields $Q\p{9n-1}$ plus a constant~\cite[A275362]{oeis}.
\item In the previous example, $Q\p{9n+8}$ has slope $10$ because unpacking it results in $10+Q\p{9n-1}$.  This appears to be, by far, the most common way steep linear solutions arise.
\end{itemize}
Our algorithm was also used to examine what sorts of recurrences can be satisfied by the exponential subsequences.  This led to the observation that any homogeneous linear recurrence with positive coefficients that sum to at least $2$ can be realized~\cite{genrusk}.


We have used our algorithm to fully explore positive recurrent solution families to the Hofstadter $Q$-recurrence with small periods.
Given a solution to the Hofstadter $Q$-recurrence, any shift of this solution is also a solution, since the recurrence only depends on the relative indices of the terms (and not the absolute indices).  So, solution families found by the algorithm can be considered as-is or modulo this shifting operation.  There are no period-$1$ positive recurrent solutions to the Hofstadter $Q$-recurrence, and there are two families (one modulo shifting) of period~$2$ solutions.  These solutions consist of one constant sequence interleaved with one standard linear sequence.  For example, the initial condition $\bk{2,2}$ gives rise to the sequence $2,2,4,2,6,2,8,2,\ldots$~\cite[A275365]{oeis}.  There are $12$ period~$3$ solution families ($4$ modulo shifting).  One of these families includes Golomb's solution, and another includes Ruskey's sequence.  The other two families consist of eventually quasilinear solutions with two constant sequences and one standard linear sequence (and appear to be related to each other).  One of these families includes sequence A264756~\cite{oeis}.  There are $12$ period~$4$ families ($5$ modulo shifting), all of which are quasilinear with constant and standard linear sequences.  Some of these families include the period~$2$ solutions, but each such family also include additional solutions that do not have period~$2$.  There are $35$ period~$5$ families ($7$ modulo shifting). Again, all of these are quasilinear.  But, one of these families has a steep linear subsequence~\cite[A269328]{oeis}.  There is a lot more variety beginning at period~$6$.  There are $294$ solution families ($86$ modulo shifting) in this case.  These solutions include quadratics as well as mixing of exponentials with steep linears.  A file containing information on all of the solution families examined, modulo shifting, can be found at \texttt{http://www.math.rutgers.edu/{$\sim$}nhf12/research/hof\_small\_periods.txt}. 

In addition, we found positive recurrent solutions to other recurrences, including the Conolly recurrence~\cite[A275363]{oeis} and the Hofstadter-Conway recurrence~\cite[A052928]{oeis}.  This second case is notable because the solution has period~$2$ with both subsequences linear.  This can happen because the Hofstadter-Conway recurrence is not basic.  As a result, the congruence classes of the constant terms in the linear polynomials end up determining much of the behavior.
\section{Future Work}\label{sec:fut}
This algorithm has various positivity requirements in order to ensure that it runs correctly.  Nested recurrences are allowed to have negative terms; these would inherently violate our algorithm's conditions.  It would be useful to have a version of this algorithm that can handle these cases.  Also, we are not currently allowing implicit solutions.  Allowing these would give a method of handling sequences with infinitely many nonpositive entries, which may or may not be helpful.

In addition, the algorithm only works for linear recurrences.  But, it makes perfect sense to look for solutions to nonlinear nested recurrences that consist of simpler interleaved sequences.  Most of the steps of this algorithm still work when the recurrences can be nonlinear.  Whether that means that $Q$ terms are multiplied by each other or raised to powers, the only thing that changes significantly is the algorithm for determining the orders of growth of the subsequences.  Nonlinearity introduces some extra complications.  For example, the recurrence $C\p{n}=C\p{n-C\p{n-1}}^2-2C\p{n-C\p{n-1}}+2$ has constant solutions (constant $1$ and constant $2$), but it also has a solution
\[
\begin{cases}
C\p{2k}=2^{2^{k-1}}+1\\
C\p{2k+1}=2.
\end{cases}
\]
The distinction here comes from the fact that repeated squaring causes numbers to rapidly grow, unless the initial number was $0$ or $1$.  Such concerns do not arise when dealing with only linear recurrences.  Hence, it would be useful to have a version of the algorithm that can handle nonlinear recurrences.  Of course, we would have to modify what we are looking for, since the example here includes a doubly exponential subsequence, which cannot possibly be a component of a positive-recurrent sequence.
\appendix
\section{Appendix: Proof of Claim~\ref{cl:main}}\label{sec:pfclmain}
In this section, we will prove Claim~\ref{cl:main}:

\clmain*

%

\begin{proof}
Since each component sequence has a rational generating function, no component sequence can possibly grow faster than exponentially.  So, proving only lower bounds when sequences should grow exponentially will suffice.  (We take advantage of this simplification throughout this proof.)

For each vertex $r$ in a directed graph $F$, define the \emph{potential} of $r$ (denoted $\phi_F\p{r}$) as the sum of the lengths of all directed cycles in $F$ containing $r$ (or $0$ if $r$ is not in any cycles).  Also, for each $r\in\bk{m}$, define $J\p{r}$ as the set of immediate successors of $r$ in $G'$.  We will prove Claim~\ref{cl:main} by induction on $\phi_G\p{r}$.  We will examine three cases (essentially a special case, a base case, and an inductive step).
\begin{description}
\item[$\phi_G\p{r}=0$:] In this case, $r$ is in no cycles in $G$, and $\bar{r}=\st{r}$.  Also, every successor of $r$ in $G$ is in $J\p{r}$, and every coefficient $\alpha_{r,\ell,j}$ with $j\notin J\p{r}$ is zero.
Inductively, suppose that any $j\in J\p{r}$ satisfies Claim~\ref{cl:main}.  (The base case of $J\p{r}=\emptyset$ is implicitly included here.)
This all gives
\[
\forab{r}{\kn}=P_r\p{\kn}+\sum\limits_{\ell=1}^{\kk}\sum\limits_{j\in J\p{r}}\alpha_{r,\ell,j} \forab{j}{\kn-\ell}.
\]
Let $D=\max\st{d_{\bar{j}}:j\in J\p{r}}$ (or $D=-\infty$ if this set is empty).  If $D=\infty$, then, by induction, some sequence $\forab{j}{\kn}$ with $j\in J\p{r}$ grows exponentially.  As a result, $\forab{r}{\kn}$ grows exponentially, and we have $d_{\bar{r}}=\infty$, as required.  On the other hand, if $D<\infty$, then we have $\forab{r}{\kn}=\Theta\pb{\kn^{\max\pb{D,\deg\p{P_r}}}}$.  Since $d_{\bar{r}}=\max\pb{D,\deg\p{P_r}}$ in this case, we have $\forab{r}{\kn}=\Theta\pb{\kn^{d_{\bar{r}}}}$, as required.
\item[$\phi_G\p{r}=1$:] In this case, the only cycle of $G$ that $r$ is in is a self-loop, and $\bar{r}=\st{r}$.  Also, every non-$r$ successor of $r$ in $G$ is in $J\p{r}$, and every coefficient $\alpha_{r,\ell,j}$ with $j\notin \st{r}\cup J\p{r}$ is zero.  Inductively, suppose that any $j\in J\p{r}$ satisfies Claim~\ref{cl:main}.  (The base case of $J\p{r}=\emptyset$ is implicitly included here.)  This all gives,
\[
\forab{r}{\kn}=P_r\p{\kn}+\sum\limits_{\ell=1}^{\kk}\alpha_{r,\ell,r}\forab{r}{\kn-\ell}+\sum\limits_{\ell=1}^{\kk}\sum\limits_{j\in J\p{r}}\alpha_{r,\ell,j} \forab{j}{\kn-\ell}.
\]
Let $D=\max\st{d_{\bar{j}}:j\in J\p{r}}$ (or $D=-\infty$ if this set is empty).  If $D=\infty$, then, by induction, some sequence $\forab{j}{\kn}$ with $j\in J\p{r}$ grows exponentially.  As a result, $\forab{r}{\kn}$ grows exponentially, and we have $d_{\bar{r}}=\infty$, as required.  On the other hand, if $D<\infty$, then we have
\[
\forab{r}{\kn}=\sum\limits_{\ell=1}^{\kk}\alpha_{r,\ell,r}\forab{r}{\kn-\ell}+\Theta\pb{\kn^{\max\pb{D,\deg\p{P_r}}}}.
\]
This is a nonhomogeneous linear recurrence for $\forab{r}{\kn}$ with characteristic polynomial
\[
P\p{x}=x^\kk-\sum\limits_{\ell=1}^{\kk}\alpha_{r,\ell,r}x^{\kk-\ell}
\]

If the weight on the self-loop on $r$ is $1$, then
\[
\sum\limits_{\ell=1}^{\kk}\alpha_{r,\ell,r}=1.
\]
So, $p\p{1}=0$, and we also note that this is a simple root, and the other nonzero roots of $p$ are all simple and roots of unity.  By the theory of linear recurrences, this results in a solution where $\forab{r}{\kn}=\Theta\pb{\kn^{1+\max\pb{D,\deg\p{P_r}}}}$.  Since $d_{\bar{r}}=1+\max\pb{D,\deg\p{P_r}}$ in this case, we have $\forab{r}{\kn}=\Theta\pb{\kn^{d_{\bar{r}}}}$, as required.

If the weight on the self-loop on $r$ is greater than $1$, then
\[
\sum\limits_{\ell=1}^{\kk}\alpha_{r,\ell,r}>1.
\]
So, $p\p{1}<0$, but $p\p{x}\to\infty$ as $x\to\infty$.  So, $p$ has a real root that is greater than $1$.  By the theory of linear recurrences, this results in a solution where $\forab{r}{\kn}$ grows exponentially.  And, $d_{\bar{r}}=\infty$ in this case, as required.
\item[$\phi_G\p{r}>1$:] Here $r$ is in some cycle of length at least $2$.  Consider such a cycle $C$.  Let $s$ denote the immediate successor of $r$, and let $t$ denote the immediate successor of $s$. We will construct a new positive recurrence system, and if we run our algorithm on this new system, we will obtain a graph in step~\ref{it:G}.  Call this graph $\tilde{G}$.  The system and $\tilde{G}$ will have the following properties:
\begin{itemize}
\item The system will have $m-1$ component sequences
\[
\forba{1}{\kn},\forba{2}{\kn},\ldots,\forba{s-2}{\kn},\forba{s-1}{\kn},\forba{s+1}{\kn},\forba{s+2}{\kn},\ldots,\forba{m}{\kn},
\]
one corresponding to each sequence $\forab{i}{\kn}$ with $i\neq s$.
\item For every $i\neq s$, if $\forba{i}{\kn}$ is given the same initial condition as $\forab{i}{\kn}$, then $\forba{i}{\kn}\leq\forab{i}{\kn}$.
\item If $s$ is in a unique cycle in $G$, then $\forba{r}{\kn}=\forab{r}{\kn}$.
\item We will have $\phi_{\tilde{G}}\p{r}<\phi_G\p{r}$.
\end{itemize}
Hence, if we inductively assume that Claim~\ref{cl:main} holds for any component sequence $r'$ of smaller potential in \emph{any} positive recurrence system, this construction will complete the proof.  (The properties are sufficient to show that  $\forab{r}{\kn}$ grows at least as fast as it should, and the equality case handles the sub-exponential component sequences.)

We have
\begin{align*}
\forab{s}{\kn}&=P_{s}\p{\kn}+\sum\limits_{\ell=1}^{\kk}\sum\limits_{j=1}^m\alpha_{s,\ell,j} \forab{j}{\kn-\ell}
\geq P_{s}\p{\kn}+\sum\limits_{\ell=1}^{\kk}\pb{\alpha_{s,\ell,t} \forab{t}{\kn-\ell}+\sum\limits_{j\in J\p{s}}\alpha_{s,\ell,j} \forab{j}{\kn-\ell}}.
\end{align*}
If $s$ is in a unique cycle in $G$, then every immediate successor of $s$ other than $t$ is in $J\p{s}$.  This means that all dropped terms in going from the equality to the inequality were in fact zero, so we actually have equality here in this case.

We also have
\begin{align*}
\forab{r}{\kn}&=P_{r}\p{\kn}+\sum\limits_{\ell=1}^{\kk}\sum\limits_{j=1}^m\alpha_{r,\ell,j} \forab{j}{\kn-\ell}
=P_{r}\p{\kn}+\sum\limits_{\ell=1}^{\kk}\pb{\alpha_{r,\ell,s} \forab{s}{\kn-\ell}+\sum\limits_{j\in\bk{m}\setminus\st{s}}\alpha_{r,\ell,j} \forab{j}{\kn-\ell}}.
\end{align*}
Let $J'=\bk{m}\setminus\p{\st{t}\cup J\p{s}}$.  Combining everything together yields
{\allowdisplaybreaks
\begin{align*}
\forab{r}{\kn}&\geq P_{r}\p{\kn}+\sum\limits_{\ell_1=1}^{\kk}\pb{\alpha_{r,\ell_1,s} \pb{P_{s}\p{\kn-\ell_1}+\sum\limits_{\ell_2=1}^{\kk}\pb{\alpha_{s,\ell_2,t} \forab{t}{\kn-\ell_1-\ell_2}+\sum\limits_{j\in J\p{s}}\alpha_{s,\ell_2,j} \forab{j}{\kn-\ell_1-\ell_2}}}\right.\\
&\hspace{0.62in}\left.+\sum\limits_{j\in\bk{m}\setminus\st{s}}\alpha_{r,\ell_1,j} \forab{j}{\kn-\ell_1}}\\
&=\pb{P_r\p{\kn}+\sum_{\ell=1}^{\kk}\alpha_{r,\ell,s}P_s\p{k-\ell}}+\sum_{\ell_1=1}^{\kk}\pb{\alpha_{r,\ell_1,t}\forab{t}{\kn-\ell_1}+\sum_{\ell_2=1}^{\kk}\alpha_{s,\ell_2,t} \forab{t}{\kn-\ell_1-\ell_2}}\\
&\hspace{0.62in}+\sum_{\ell_1=1}^{\kk}\sum_{j\in J\p{s}}\pb{\alpha_{r,\ell_1,j}\forab{j}{\kn-\ell_1}+\sum_{\ell_2=1}^{\kk}\alpha_{s,\ell_2,j} \forab{j}{\kn-\ell_1-\ell_2}}+\sum_{\ell=1}^{\kk}\sum_{j\in J'}\alpha_{r,\ell,j}\forab{j}{\kn-\ell}.
\end{align*}
Again, if $s$ is in a unique cycle in $G$, we actually have equality here.}

Our new positive recurrence system is obtained from the current one as follows:
\begin{itemize}
\item The equation for $\forba{i}{\kn}$ for $i\notin\st{r,s}$ is
\[
\forba{i}{\kn}=P_i\p{\kn}+\sum\limits_{\ell=1}^{\kk}\sum\limits_{j\in\bk{m}\setminus\st{s}}\alpha_{i,\ell,j} \forba{j}{\kn-\ell}.
\]
\item We have
\begin{align*}
\forba{r}{\kn}&=\pb{P_r\p{\kn}+\sum_{\ell=1}^{\kk}\alpha_{r,\ell,s}P_s\p{k-\ell}}+\sum_{\ell_1=1}^{\kk}\pb{\alpha_{r,\ell_1,t}\forba{t}{\kn-\ell_1}+\sum_{\ell_2=1}^{\kk}\alpha_{s,\ell_2,t} \forba{t}{\kn-\ell_1-\ell_2}}\\
&\hspace{0.62in}+\sum_{\ell_1=1}^{\kk}\sum_{j\in J\p{s}}\pb{\alpha_{r,\ell_1,j}\forba{j}{\kn-\ell_1}+\sum_{\ell_2=1}^{\kk}\alpha_{s,\ell_2,j} \forba{j}{\kn-\ell_1-\ell_2}}+\sum_{\ell=1}^{\kk}\sum_{j\in J'}\alpha_{r,\ell,j}\forba{j}{\kn-\ell}.
\end{align*}
\end{itemize}
In the definition of $\forba{r}{\kn}$, the first part is the nonhomogeneous part: a polynomial of degree $\max\pb{\deg\pb{P_r},\deg\pb{P_s}}$.  The second part describes the references to $\fb{t}$.  The third part describes the references to $\fb{j}$ where $j\in J\p{s}$.  The final part describes all the other references.  (Note that there are no references to the nonexistent sequence $\fb{s}$.)

From this construction and our calculation, we notice that the first three desired properties of this new system definitely hold.  We can also see how $\tilde{G}$ is obtained from $G$:
\begin{itemize}
\item First, delete all incoming arcs to $s$, other than the one from $r$.  (This is accomplished by the first bullet point above.)
\item Then, contract the arc $rs$.  (This is accomplished by the second bullet point above.)
\item If any multiple arcs were created in the previous step, replace them by a single arc whose weight is the sum of the weights of the arcs being replaced.  (This would come up, for example, if $r$ has a self-loop and is also in a $2$-cycle, with $s$ being the other vertex in that cycle.  After contracting $rs$, $r$ would have two self-loops, which then need to be combined.)
\end{itemize}
From this description, it is clear that every cycle in $G$ containing $r$ corresponds to a cycle in $\tilde{G}$ containing $r$, and the latter cycle can be no longer than the original.  The cycle corresponding to $C$ in $\tilde{G}$ is strictly shorter than $C$, since $s$ has been removed from it.  So, $\phi_{\tilde{G}}\p{r}<\phi_G\p{r}$, as required.
\end{description}
\end{proof}

\section*{Acknowledgements}
I would like to thank Dr. Doron Zeilberger of Rutgers University for introducing me to the work of Golomb and Ruskey and suggesting I try to generalize their work.  I would also like to thank Anthony Zaleski of Rutgers University for proofreading a draft of this paper and providing me with useful feedback.


\begin{bibdiv}
\begin{biblist}
\bib{hofv}{article}
{
  title={On the behavior of a variant of Hofstadter's Q-sequence},
  author={Balamohan, B.},
  author = {Kuznetsov, A.},
  author = {Tanny, Stephen},
  journal={J. Integer Sequences},
  volume={10},
  pages={29},
  year={2007}
}
\bib{con}{misc}
{
author={Conolly, B.W.},
title={Meta-Fibonacci sequences, Chapter XII in S. Vajda, Fibonacci \& Lucas Numbers, and the Golden Section},
year={1989},
publisher={Ellis Horwood Limited}
}
\bib{erickson}{article}
{
  title={Nested recurrence relations with Conolly-like solutions},
  author={Erickson, Alejandro}
  author={Isgur, Abraham}
  author = {Jackson, Bradley W.}
  author = {Ruskey, Frank}
  author = {Tanny, Stephen M.},
  journal={SIAM Journal on Discrete Mathematics},
  volume={26},
  number={1},
  pages={206--238},
  year={2012},
  publisher={SIAM}
}
\bib{genrusk}{article}
{
  title={Linear recurrent subsequences of generalized meta-Fibonacci sequences},
  author={Fox, Nathan},
  journal={Journal of Difference Equations and Applications},
  pages={1--8},
  year={2016},
  publisher={Taylor \& Francis}
}
\bib{gengol}{article}
{
 title={Quasipolynomial Solutions to the Hofstadter Q-Recurrence},
 author={Fox, Nathan},
 journal={arXiv preprint arXiv:1511.06484},
 year={2015}
}
\bib{golomb}{misc}
{
author={Golomb, S.W.},
title={Discrete Chaos: Sequences Satisfying \quot{Strange} Recursions},
year={1991},
publisher={unpublished manuscript}
}
\bib{geb}{book}
{
 author = {Hofstadter, Douglas},
 title = {G\"odel, Escher, Bach: an Eternal Golden Braid}, 
 publisher = {Penguin Books},
 year = {1979}, 
 pages = {137}
}
\bib{isgur1}{article}
{
  title={Constructing New Families of Nested Recursions with Slow Solutions},
  author={Isgur, A.}
  author ={Lech, R.}
  author = {Moore, S.}
  author = {Tanny, S.}
  author = {Verberne, Y.}
  author = {Zhang, Y.},
  journal={SIAM Journal on Discrete Mathematics},
  volume={30},
  number={2},
  pages={1128--1147},
  year={2016},
  publisher={SIAM}
}
\bib{isgur2}{article}
{
  title={Trees and meta-Fibonacci sequences},
  author={Isgur, Abraham}
  author={Reiss, David}
  author = {Tanny, Stephen},
  journal={The Electronic Journal of Combinatorics},
  volume={16},
  number={R129},
  pages={1},
  year={2009}
}
\bib{karp}{article}
{
  title={Reducibility among combinatorial problems},
  author={Karp, Richard M.},
  booktitle={Complexity of computer computations},
  pages={85--103},
  year={1972},
  publisher={Springer}
}
\bib{mallows}{article}
{
  title={Conway's challenge sequence},
  author={Mallows, Colin L},
  journal={American Mathematical Monthly},
  volume={98},
  number={1},
  pages={5--20},
  year={1991},
  publisher={Mathematical Association of America}
}
\bib{oeis}{misc}
{
 title = {OEIS Foundation Inc.},
 year = {2016}
 publisher = {The On-Line Encyclopedia of Integer Sequences}
 note = {http://oeis.org/}
}
\bib{rusk}{article}
{
author = {Ruskey, F.}
title = {Fibonacci Meets Hofstadter},
journal={The Fibonacci Quarterly},
volume= {49},
year = {2011},
number = {3},
pages = {227-230}
}
\bib{tanny}{article}
{
title={A well-behaved cousin of the Hofstadter sequence},
  author={Tanny, Stephen M},
  journal={Discrete Mathematics},
  volume={105},
  number={1},
  pages={227--239},
  year={1992},
  publisher={Elsevier}
}
\end{biblist}
\end{bibdiv}
\end{document}